\documentclass[a4paper,reqno,12pt]{amsart}

\usepackage[utf8]{inputenc}
\usepackage{amsmath, amssymb, amsthm, epsfig}
\usepackage{hyperref, latexsym}
\usepackage{url}
\usepackage[mathscr]{euscript}
\usepackage{color}
\usepackage{harpoon}
\usepackage{url}
\usepackage{mathabx}
\usepackage{tikz}
\usepackage{mathtools,fbox}
\usepackage{refcheck}
\DeclareMathAlphabet{\pazocal}{OMS}{zplm}{m}{n}

\usepackage{pgfplotstable}
\usepackage{pgfplots}

\usepackage{fullpage} 
\usepackage{setspace}

\usepackage{mathtools}
\mathtoolsset{showonlyrefs}

\usepackage{refcheck} 
\norefnames
\nocitenames

\def\today{\ifcase\month\or
  January\or February\or March\or April\or May\or June\or
  July\or August\or September\or October\or November\or December\fi
  \space\number\day, \number\year}

\DeclareMathOperator{\supp}{\mathrm{supp}}

\newtheorem{theorem}{Theorem}

\newtheorem{lemma}[theorem]{Lemma}
\newtheorem{proposition}[theorem]{Proposition}
\newtheorem{corollary}[theorem]{Corollary}

\newtheorem{definition}{Definition}

\newcommand{\A}{\mathcal{A}}

\newcommand{\C}{\mathcal{C}}

\newcommand{\N}{\mathcal{N}}

\renewcommand{\P}{\mathcal{P}}

\renewcommand{\S}{\mathcal{S}}

\newcommand{\NN}{\mathfrak{N}}

\newcommand{\n}{\mathbb{N}}
\newcommand{\z}{\mathbb{Z}}
\newcommand{\q}{\mathbb{Q}}
\renewcommand{\r}{\mathbb{R}}
\newcommand{\cp}{\mathbb{C}} 
\newcommand{\im}{{\rm Im}\,}
\newcommand{\re}{\Re}
\newcommand{\ft}{\widehat}

\newcommand{\bo}{\boldsymbol}

\newcommand{\wt}{\widetilde}

\newcommand{\la}{\lambda}
\newcommand{\ga}{\gamma}
\newcommand{\al}{\alpha}

\newcommand{\ep}{\varepsilon}

\newcommand{\p}{\varphi}

\renewcommand{\d}{\mathrm{d}}
\newcommand{\ov}[1]{\overline{#1}}
\newcommand{\w}{\ov{w}}
\newcommand{\cd}{\cdot}
\newcommand{\1}{{\bf 1}}
\newcommand{\ra}{\rightarrow}

\newcommand{\e}{\mathbb{E}}
\newcommand{\mt}{\mapsto}
\newcommand{\del}{{\bo \delta}}

\renewcommand{\Re}{{\rm Re}\,}
\renewcommand{\Im}{{\rm Im}\,}
\newcommand{\res}[1]{ \mathop{\mathrm{Res}}_{z = #1}\,}
\newcommand{\spec}{{\rm spec}}

\begin{document}


\title[]{A Complete Classification of Fourier Summation Formulas on the real line}
\author{Felipe Gonçalves and Guilherme Vedana}
\date{\today}
\subjclass[2010]{}
\address{The University of Texas at Austin, 2515 Speedway, Austin, TX 78712, USA  \newline   {\&} IMPA - Instituto de Matemática Pura e Aplicada, Rio de Janeiro, 22460-320, Brazil.}
\email{goncalves@utexas.edu}
\address{IMPA - Estrada Dona Castorina 110, Rio de Janeiro, RJ - Brasil, 
22460-320
}
\email{guilherme.israel@impa.br}
\allowdisplaybreaks


\begin{abstract}
We completely classify Fourier summation formulas of the form
$$
 \int_{\r} \ft{\varphi}(t) d\mu(t)=\sum_{n=0}^{\infty} a(\la_n)\varphi(\la_n),
$$
that hold for any test function $\p$, where $\ft \p$ is the Fourier transform of $\p$, $\mu$ is a fixed complex measure on $\r$ and $a:\{\la_n\}_{n\geq 0}\to\cp$ is a fixed function. We only assume the decay condition
$$
\int_{\r} \frac{d |\mu|(t)}{(1+t^2)^{c_1}} + \sum_{n\geq 0} |a(\la_n)|e^{-c_2 |\la_n|}<\infty,
$$
for some $c_1,c_2>0$. This completes the work initiated by the first author previously, where the condition $c_1\leq 1$ was required. We prove that any such pair $(\mu,a)$ can be uniquely associated with a holomorphic map $F(z)$ in the upper-half space that is both almost periodic and belongs to a certain higher index Nevanlinna class. The converse is also true: For any such function $F$ it is possible to generate a Fourier summation pair $(\mu,a)$. We provide important examples of such summation formulas not contemplated by the previous results, such as Selberg's trace formula.
\end{abstract}


\maketitle



\section{Introduction}
{Fourier Summation pairs} (FS-pairs, for short) play a key role in the study of several questions in Number Theory, from geometry of numbers to analytic number theory. The importance of such formulas relies on the fact that they establish a critical relation between two important quantities: Phase and frequency information. The  probably most well-known  example of a FS-pair is given by the {Poisson Summation formula}
\begin{equation}
    \sum_{v\in\z} \p(v)=\sum_{u\in\z} \ft{\p}(u) ,
\end{equation}
that holds for any test\footnote{In this paper, a test function $\p:\r\to\cp$ is always $C^\infty$ and compactly supported.} function $\p$. Throughout this paper we use the following normalization of the Fourier transform
\begin{align*}
\ft{f}(\xi)=\int_{\r^d} f(x)e^{-2\pi ix\cd\xi}dx.
\end{align*}
Poisson summation and its higher dimensional versions is a fundamental tool in Number Theory, with countless applications. A FS-pair is a generalization of the Poisson Summation formula: It is a pair $(\mu,a)$, where $\mu$ is a complex measure on $\r$ and $a:\r\rightarrow\cp$ is a function with countable support, such that
\begin{equation}\label{protofspair}
    \int_{\r} \ft{\varphi}(t) d\mu(t)=\sum_{n=0}^{\infty} a(\la_n)\varphi(\la_n)
\end{equation}
holds for any test function $\p$, where $\{\la_n\}_{n=0}^\infty=\supp(a)$ is some enumeration of the support of $a(\cd)$. 

There are plenty of examples of FS-pairs in Number Theory. The {Guinand-Weil explicit formula} (see, for instance, \cite[Thm. 12.13]{MV}) and, more generally, the {Selberg Trace formula} (see \cite[Thm. 5.6]{Ber}), all generate FS-pairs. Furthermore, the recent Radchenko-Viazovska interpolation formulas in \cite{RV} can also be seen as FS-pairs (see also \cite{RS,KNS}). For a more recent account of different kinds of FS-pairs see the last section in \cite{G23} (see also Section \ref{sec:exe}). These summation formulas also appear in {Physics}. The special class where the support of both $\mu$ and $a(\cd)$ are locally finite, which is called a {crystalline} pair, is used in {Crystallography} for the reconstruction of the atomic structure of {crystals}. If the structure of the crystal is represented by a measure $\mu$ (this is unknown, a priori), then diffraction experiments provide the values of $|a|^2$ from which one can numerically recover $\mu$. Therefore, a complete characterization of the crystalline pairs (and more generaliy FS-pairs) can be a useful tool.

The main objective of this paper is to finish the classification initiated by the first named author in \cite{G23}, where formulas of the type \eqref{protofspair} were classified under the assumption that $\int_{\r} (1+t^2)^{-1}\d|\mu|(t)<\infty$. Here, we found a way to circumvent the issues in \cite{G23} and allow any polynomial growth for $\mu$, that is, we only require that $\int_{\r} (1+t^2)^{-c}\d|\mu|(t)<\infty$ for some $c>0$. The results of this paper are heavily inspired and motivated by recent work  in the field, such as: The new summation formulas produced by Radchenko and Viazovska \cite{RV}, Bondarenko, Radchenko and Seip \cite{BRS}, Kurasov and Sarnak \cite{KS}, Kulikov, Nazarov and Sodin \cite{KNS} and Ramos and Sousa \cite{RS,RS1}; The recent results on Lee-Yang polynomials of Alon, Cohen and Vinzant \cite{ACV,AV}; The classification of crystalline pairs of Lev and Olevskii  \cite{LO,LO2} and Olevskii and Ulanovskii \cite{OU}.

\section{Main results}

In order to state our main results we will need some preparatory definitions. We follow the notation in \cite{G23} closely. We say that a complex-valued Borel measure $\mu$ on $\r$ is {strongly tempered} if
\begin{align*}
\deg(\mu):=\inf \left\{n\in\z; \int_{\r} \frac{d|\mu|(t)}{(1+t^2)^{n/2}}<\infty\right\} < \infty.
\end{align*}
In this case, the map $\p\in\S(\r)\mt\int_{\r} \p(t) d\mu(t)$ defines a tempered distribution. A function $a:\r\ra\cp$ is called {\bf locally summable} if its support $\supp(a):=\{\la\in\r;a(\la)\neq0\}$ is countable, and for some (and hence for all) enumeration $\supp(a)=\{\la_n;n\geq0\}$, the sum
\begin{align*}
\sum_{n\geq0;\,\la_n\in[-T,T]} |a(\la_n)| <\infty
\end{align*}
for any $T>0$.  In addition, we say that $a(\cd)$ has {\bf finite exponential growth} if
\begin{align*}
\sum_{n\geq0} |a(\la_n)|e^{-c|\la_n|}<\infty,
\end{align*}
for some $c>0$. In particular the sum  $\sum_{\la\in \r} a(\la)\p(\la)$ is well-defined for any function $\p$ such that $|\p(\la)| \lesssim e^{-c|\la|}$ (for instance if $\p$ is a test function). At this point we are able to define what a {Fourier summation pair} is.

\begin{definition}[Fourier summation pair]
A {Fourier summation pair} (FS-pair) is a pair $(\mu,a)$, where $\mu$ is a strongly tempered measure on $\r$, $a:\r\ra\cp$ is a locally summable function, such that
\begin{align}\label{def:fspair}
\int_{\r} \ft{\p}(t)d\mu(t)=\sum_{\la\in\r} a(\la)\p(\la)
\end{align}
holds for any test function $\p$.
\end{definition}
In view of this definition, the decay condition
$$
\int_{\r} \frac{d |\mu|(t)}{(1+t^2)^{c_1}} + \sum_{n\geq 0} |a(\la_n)|e^{-c_2 |\la_n|}<\infty,
$$
for some $c_1,c_2>0$, seems to be the most weak assumption one can impose. In other words, that $\mu$ has finite degree and $a(\cdot)$ has finite exponential growth. Moreover, to the best of our knowledge, we are unaware of any summation formula as \eqref{def:fspair} that holds for every smooth compactly supported function $\p$, but does not satisfy the above decay condition. This is the only decay condition our main result Theorem \ref{thm:main} is going to assume, which makes it quite general.

In order to avoid technicalities, we will state our results for FS-pairs $(\mu,a)$ for which $\mu$ is a real measure. This implies automatically that $a(-\la)=\ov{a(\la)}$ for any $\la$, a property that we call {antipodal} and the pair $(\mu,a)$ is then called {\bf real-antipodal}. Given any FS-pair $(\mu,a)$, we can split it into two real-antipodal FS-pairs $(\mu_1,a_1)$ and $(\mu_2,a_2)$ by defining $\mu_1=\Re(\mu)$, $\mu_2=-\Im(\mu)$, $a_1(\la):=(a(\la)+\ov{a(-\la)})/2$ and $a_2(\la):=-i(\ov{a(-\la)}-a(\la))/2$, so that $a=a_1-ia_2$.

\subsection*{Almost Periodic Class} Let $\cp^+:=\{z\in\cp;\Im(z)>0\}$ be the complex upper-half plane. A holomorphic map $F(z)$ defined on $\cp^+$ is said to be {almost periodic} if, for any $0<\al<\beta<\infty$ and $\ep>0$, there exists a relatively dense\footnote{This means that there exists $l>0$ such that $\tau\cap(x,x+l)\neq\emptyset$ for any $x\in\r$.} set of translations $\tau\subset\r$, which may depend on $\al,\beta$ and $\ep$, such that
\begin{align*}
\sup_{\al<\Im(z)<\beta} |F(z+t)-F(z)|\leq\ep,\,\,\,\, \text{ for any } t\in\tau.
\end{align*}
We denote this class of almost periodic functions by $\A\P(\cp^+)$.
In this case we can define an analogous of a Fourier coefficient for $F(z)$: For any $\la\in\r$, the limit
\begin{align*}
\e F(\la):=\lim_{T\ra\infty} \frac{1}{2T} \int_{-T+iy}^{T+iy} F(z)e^{-2\pi i\la z}dz
\end{align*}
does exist and does not depend on $y>0$. In particular, if $F(z+ic)$ is almost periodic, then the above limit exists and is independent of $y>c$.  If $F(z+ic)\in\A\P(\cp^+)$, its {spectrum} is defined by
\begin{align*}
\spec(F):=\{\la\in\r;\e F(\la)\neq0\},
\end{align*}
and it is a countable set. For more information about almost periodic functions, see for instance \cite{Be,Bo,G23}.

\subsection*{The Nevanlinna Class}The higher order holomorphic Nevanlinna class $\NN_{\leq k}$ is defined as the set of all holomorphic maps $F:\cp^+\ra\cp$ such that for any choice of $z_1,...,z_N \in \cp^+$, the Hermitian matrix
\begin{align*}
\left[i\frac{F(z_n)+\ov{F(z_m)}}{z_n-\ov{z_m}}\right]_{1\leq n,m\leq N}
\end{align*}
has at most $k$ negative eigenvalues (counting multiplicities). 
We define
$$
\NN_{\leq k}-\NN_{\leq k}:=\{F-G;F,G\in\NN_{\leq k}\}.
$$
For more information, see Section 3. 

The following is the main result of this paper. It states that if we have a FS-pair $(\mu,a)$, then we can associate a holomorphic function $F(z)$ in $\cp^+$ that is at the same time almost periodic and belongs to the class $\NN_{\leq k}-\NN_{\leq k}$. This function $F(z)$ encapsulates the information contained in the pair $(\mu,a)$ by having the coefficients of its ``Fourier series'' given by the function $a(\cd)$ and $\mu$ is the measure from its Nevanlinna factorization. The converse is also true: starting from any such function $F(z)$, it is possible to build a FS-pair. In short, the results in the paper answer the following question: 

\emph{What is a FS-pair? It is a function in the following class}
$$
\bigg(\bigcup_{k\in \z_+} ( \NN_{\leq k}-\NN_{\leq k})\bigg) \cap \bigg(\bigcup_{c\in \r_+} \A\P(\cp^++ic)\bigg).
$$
\begin{theorem}[Classification of FS-pairs]
\label{thm:main}
Let $(\mu,a)$ be a real-antipodal FS-pair such that $a(\cd)$ has finite exponential growth and that $\deg(\mu)\leq 2(k+1)$. Then, to the pair $(\mu,a)$ corresponds a unique holomorphic map $F(z)$ in $\cp^+$ which satisfies the following properties:

\begin{itemize}
\item[(I)] $F(z) \in \NN_{\leq k}-\NN_{\leq k}$;

\item[(II)] $F(\cd+ic_1)\in \A\P(\cp^+)$ for some $c_1>0$;

\item[(III)] $\la \mapsto \e F(\la)$ is a function of finite exponential growth, this is, $\sum_{\la\in\r} |\e F(\la)|e^{-c_2|\la|}<\infty$, for some $c_2>0$.

\end{itemize}
This function $F$ is given by the following identities
\begin{align}\label{id:Finthm}
F(z)= \frac{(z^2+1)^k}{2\pi i} \int_{\r} \frac{1+tz}{t-z}\cd\frac{d\mu(t)}{(1+t^2)^{k+1}}+iQ(z) = \frac{1}{2}a(0)+\sum_{\la>0} a(\la)e^{2\pi i\la z},
\end{align}
where $Q(z)$ is a real polynomial of degree $\leq 2k$. The first identity above holds for all $z\in \cp^+$, while the second only if $\Im z>c_1$. Both expressions converge absolutely in these domains.

Conversely, if $F(z)$ is a holomorphic map in $\cp^+$ satisfying properties (I),(II) and (III), then we can construct a real-antipodal FS-pair $(\mu,a)$ where $a(\cd)$ has finite exponential growth and $\deg(\mu)\leq 2(k+1)$. More precisely, there exists $c_1>0$ such that the limit
\begin{equation}
\e F(\la):=\lim_{T\ra\infty} \frac{1}{2T}\int_{-T+iy}^{T+iy} F(z)e^{-2\pi i\la z}dz
\end{equation}
does exist for all $\la\in\r$ and $y>c_1$, and does not depend on $y$. Moreover, the function $\la\mt\e F(\la)$ vanishes for $\la<0$ and has finite exponential growth. The function $a(\cdot)$ is defined by
\begin{equation}
\label{eq:def_function_a}
a(\la):=\begin{cases}
\e F(\la) & \text{if } \la>0,\\
2\Re \e F(0) & \text{if } \la=0,\\
\ov{\e F(-\la)} & \text{if } \la<0.
\end{cases}
\end{equation}
The measure $\mu$ is the unique real-valued measure coming from the Nevanlinna factorization \eqref{id:nevfact} of $F(z)$.
\end{theorem}

The polynomial $Q$ is uniquely defined by identity \eqref{id:Finthm} and it has degree $2k$ (and not  $2k+1$ as in \eqref{id:nevfact}) since $$
\lim_{y\to \infty}\re \frac{F(iy)-iQ(iy)}{y(1-y^2)^k} = \lim_{y\to \infty} \int_\r \frac{\d \mu(t)}{2\pi(y^2+t^2)(1+t^2)^k} = 0,
$$
and $F(iy)\to \frac12 a(0)$ as $y\to \infty$.

Regarding the converse, it turns out that we can weaken the assumption of finite exponential growth of the Fourier coefficients $\e F(\la)$ (property (III)) and just require local summability. In this case we can also construct a FS-pair $(\mu,a)$, though now the function $a(\cd)$ is no longer necessarily of finite exponential growth.

\begin{theorem}
\label{thm:converse_main}
Let $F \in \NN_{\leq k}-\NN_{\leq k}$ be such that $F(\cd+ic)\in \A\P(\cp^+)$ for some $c>0$ and assume that the function $\la\in\r\mt\e F(\la)$ is locally summable. Then one can associate a real-antipodal FS-pair $(\mu,a)$ exactly as in Theorem \ref{thm:main}, except now that $a(\cdot)$ is only locally summable. 
\end{theorem}

\section{The Nevanlinna class}\label{sec:neva}
All facts described in this section can be found in \cite{DL,KL,KW}.
The {Nevanlinna class of index $k=0$} is the set of all holomorphic maps $F:\cp^+\ra\cp$ such that $\Re F(z)\geq0$ for any $z\in\cp^+$. It can be shown that this is equivalent to $F$ having the following {Poisson representation} \cite[Thm 4]{dB}
\begin{equation}
\label{eq:Nevanlinna_fact_0}
F(z)=iQ(z)+\frac{1}{2\pi i}\int_{\r} \frac{1+tz}{t-z}\cd\frac{d\mu(t)}{1+t^2},
\end{equation} 
where $\mu$ is a nonnegative Borel measure on $\r$ of degree $\deg(\mu)\leq 2$ and $Q(z)=a+bz$ with $b\leq 0$. The {(generalized) Nevanlinna class of index $\leq k$} is defined as the set of all meromorphic maps $F:\cp^+\ra\cp$ such that, for any choice $z_1,...,z_N\in \cp^+$, the matrix
\begin{align*}
\left[i\frac{F(z_n)+\ov{F(z_m)}}{z_n-\ov{z_m}}\right]_{1\leq n,m\leq N}
\end{align*}
has at most $k$ negative eigenvalues.  The next proposition is a restatement of  \cite[Prop. 2.1]{DL} for the holomorphic scalar case (see also \cite[eq. (4)]{DLLS})

\begin{proposition}[]\label{prop:neva}
Let $F:\cp^+\ra\cp$ be a holomorphic function in the generalized Nevanlinna class of index $\leq k$. Then it is possible to write
\begin{equation}
\label{eq:Nevanlinna_fact_k}
F(z)=\frac{(z^2+1)^{m}}{2\pi i}\int_{\r} \frac{tz+1}{t-z}\cd\frac{d\mu(t)}{(1+t^2)^{m+1}}+iQ(z),
\end{equation}
where $\mu$ is a nonnegative Borel measure on $\r$ such that $\deg(\mu)\leq 2(m+1)$and $Q(z)=a_{2m+1}z^{2m+1}+...+a_1z+a_0$ is a real polynomial of degree $\leq 2m+1$ such that $a_{2m+1}\leq 0$.  Conversely, any function defined by \eqref{eq:Nevanlinna_fact_k} with $m\leq k$ (and the same constraints on $\mu$ and $Q$) defines a holomorphic function in the generalized Nevanlinna class of index $\leq k$.
\end{proposition}
We denote the above class by $\NN_{\leq k}$. Moreover, in order to account for the cases in which $\mu$ is a signed measure, we also consider the class $\NN_{\leq k}-\NN_{\leq k}:=\{F-G;F,G\in\NN_{\leq k}\}$. We  can then apply Proposition \ref{prop:neva} to obtain that $F\in \NN_{\leq k}-\NN_{\leq k}$ if and only if
\begin{align}\label{id:nevfact}
F(z)=\frac{(z^2+1)^m}{2\pi i}\int_{\r} \frac{tz+1}{t-z}\cd\frac{d\mu(t)}{(1+t^2)^{m+1}}+iQ(z),
\end{align}
for $m\leq k$, a signed measure $\mu$  with $\deg(\mu)\leq 2(m+1)$, and $Q(z)$, a real polynomial of degree at most $2m+1$. 

We now make some few remarks about the class $\NN_{\leq k}-\NN_{\leq k}$. Firstly, we note that $F\in \NN_0-\NN_0$ if and only if one can write $e^F = P/Q$, where $P$ and $Q$ are holomorphic and bounded by $1$ in $\cp^+$ (this is the Bounded Type class used in \cite{G23}, see also \cite[Thm. 9]{dB}). Secondly, we note that $Q$ is uniquely defined in terms of $F$. Indeed, we have
$$
\Re \frac{F(iy)-iQ(iy)}{(1-y^2)^m y} = \int_\r \frac{1}{y^2+t^2}\frac{\d \mu(t)}{(1+t^2)^m} \to 0
$$
as $y\to \infty$, and since $Q$ has degree at most $2m+1$, we conclude that
$$
\lim_{y\to \infty} y^{-2m-1}\re F(iy)=(-1)^{m+1}a_{2m+1}.
$$
Moreover, since $(F(z)-iQ(z))/(z^2+1)^m$ is holomorphic and $Q$ can only have real coefficients, we must have
$$
Q(z) = -i(z^2+1)^m(R(1/(z-i)) - R^*(1/(z+i)) + a_{2m+1}z(z^2+1)^m
$$
where $R(1/(z-i))$ is the singular part of the Laurent expansion of $F(z)/(z^2+1)^m$ at $z=i$. To see this, first observe there always exist a polynomial $R$ of degree at most $m$ such that we can write $Q$ in above form. If we momentarily let $\text{Sing}[\cdot]$ be the singular part of a given function at $z=i$, we obtain
\begin{align*}
0=\text{Sing}[(F(z)-iQ(z))/(z^2+1)^m] =  \text{Sing}[F(z)/(z^2+1)^m - R(1/(z-i))].
\end{align*}
Thirdly, $\mu$ is also uniquely determined by $F$ since a routine argument shows that
$$
\lim_{s\searrow 0} \Re \int_{a+is}^{b+is} \frac{F(z)-iQ(z)}{(z^2+1)^k}\d z = \int_{a}^b \frac{\tfrac12 d\mu(t)}{(1+t^2)^{k+1}},
$$
whenever $a<b$ are points of continuity of $\mu$. Finally, we note that (and this is going to be extremely useful later on), whenever \eqref{id:nevfact} holds, then for any $y_0>0$ we also have the representation
\begin{align}\label{id:nevfacty0}
F(z)=\frac{(z^2+y_0^2)^m}{2\pi i}\int_{\r} \frac{tz+y_0^2}{t-z}\cd\frac{d\mu(t)}{(y_0^2+t^2)^{m+1}}+iQ_{y_0}(z),
\end{align}
where $Q_{y_0}(z) = (z^2+y_0^2)^m(R_{y_0}(1/(z-iy_0)) + R_{y_0}^*(1/(z+iy_0)) + a_{2m+1}z(z^2+y_0^2)^m$ and $R_{y_0}(1/(z-iy_0)) $ is the singular part of $F(z)/(z^2+y_0^2)^m$ at $z=iy_0$. The form of $Q_{y_0}$ can be derived (as before) by similar considerations from the above integral identity. Hence, to show the above identity holds true, it is enough to prove that if we let $\wt F(z)$ denote the holomorphic function on the right hand side of \eqref{id:nevfacty0} then
$P(z):=i(F(z)-\wt F(z))$ is a real polynomial of degree at most $2m+1$. To this end we make use of the identity
$$
\frac{(z^2+r^2)^m(r^2+tz)}{(r^2+t^2)^{m+1}(t-z)} = \frac{1}{t-z} - \frac{t+z}{r^2+t^2}\sum_{j=0}^{m-1} \bigg(\frac{z^2+r^2}{r^2+t^2}\bigg)^j - \frac{t(r^2+z^2)^m}{(r^2+t^2)^{m+1}}.
$$
Taking the difference of two such identities for $r=1$ and $r=y_0$, we conclude that $P(z)=i(F(z)-\wt F(z))$ extends to an real entire function.

Similarly, $ \lim_{y\to \infty} y^{-2m-1}\wt F(iy)$ also exists. Since both $F$ and $\wt F$ are of Bounded Type\footnote{Functions of bounded type in $\cp^+$ form an algebra that contains polynomials and functions with nonnegative real part, hence $F$ and $\wt F$ are of Bounded Type.} in $\cp^+$, a classical result of Krein \cite[Prob. 37]{dB} shows that $F-\wt F$ must have finite exponential type. However, since $|P(iy)|=O(|y|^{2m+1})$, we conclude that $P$ must be a real polynomial of degree at most $2m+1$.

\section{Preliminaries}

In order to prove the main theorem, we follow an analogous of the strategy used in \cite{G23}. We begin by defining some auxiliary functions.  For $z,w\in\cp^+$, $x\in \r$ and $k\in \z_+$ we define
\begin{align*}
G_0(w,z,x)& :=\frac{e^{-2\pi i\w|x|}\1_{x<0}+e^{2\pi iz|x|}\1_{x\geq0}}{z-\w} \\
A_k(x)&:=\underbracket{e^{-2\pi|\cd|}*\cdots * e^{-2\pi|\cd|}(x)}_{k \text{-times}}  \\
G_k(w,z,x)&:=G_0(w,z,\cd)* A_k(x) \quad (k\geq 1).
\end{align*}

\begin{lemma}[Properties of auxiliary functions]
\label{thm:prop_aux_func}
The above functions have the following properties
\begin{itemize}
\item [(i)] We have 
\begin{equation}
\ft{G_k}(w,z,t)=\frac{1}{2\pi^{k+1}i}\cd\frac{1}{(t-z)(t-\w)}\cd\frac{1}{(1+t^2)^k};
\end{equation}
where the Fourier transform is taken in the last variable.
\item[(ii)] There exist polynomials $r_{k} \in \q[X]$ of degree exactly $k$ such that
\begin{equation}
A_k(x)=e^{-2\pi|x|}\pi^{1-k} r_{k-1}(2\pi |x|).
\end{equation}
These polynomials have the following generating series
\begin{align*}
 \sum_{k\geq 0} q^k r_{k}(X)  =\frac{e^{(1-\sqrt{1-q})X}}{\sqrt{1-q}} = 1
 + \left(\frac{1}{2}\*X
 + 1\right) \*q
 + \left(\frac{1}{8}\*X^2
 + \frac{5}{8}\*X
 + 1\right) \*q^2+ O(q^3),
\end{align*}
which converges absolutely for $|q|<1 $ and $X\in \r$;
\item[(iii)] For $\la\geq0$ we have $G_k(w,z,-\la)=-\ov{G_k(z,w,\la)}$, and if we write $p_{k-1}(x):=\pi^{1-k} r_{k-1}(2\pi x)=\sum_{j=0}^{k-1} b_{k-1,j}x^j$, then
\begin{align*}
(z-\w)G_k(w,z,\la)&=e^{-2\pi\la}\sum_{j=0}^{k-1} \frac{j!b_{k-1,j}}{(2\pi)^{j+1}(1+i\w)^{j+1}}\sum_{l=0}^{j} \frac{(2\pi\la)^l(1+i\w)^l}{l!}\\
&-e^{-2\pi\la}\sum_{j=0}^{k-1} \frac{j!b_{k-1,j}}{(2\pi)^{j+1}(1+iz)^{j+1}}\sum_{l=0}^{j} \frac{(2\pi\la)^l(1+iz)^l}{l!}\\
&+e^{2\pi i\la z}\sum_{j=0}^{k-1} \frac{j!b_{k-1,j}}{(2\pi)^{j+1}}\left[\frac{1}{(1+iz)^{j+1}}+\frac{1}{(1-iz)^{j+1}}\right].
\end{align*}
If $z=i$ one should take the limit in the above expression. Moreover, for fixed $w\in\cp^+$, the map $z\in\cp^+\mapsto G_k(w,z,\la)$ is holomorphic. The same is true in the variable $w$ if we fix $z$.
\item[(iv)] For $z\neq\pm i$:
\begin{equation}
\sum_{j=0}^{k-1} \frac{j!b_{k-1,j}}{(2\pi)^{j+1}}\left[\frac{1}{(1+iz)^{j+1}}+\frac{1}{(1-iz)^{j+1}}\right]=\frac{1}{\pi^k}\cd\frac{1}{(1+z^2)^k}.
\end{equation}
Hence, for $z\neq i$, we can write
\begin{align*}
(z-\w)G_k(w,z,\la)&=e^{-2\pi\la}\sum_{j=0}^{k-1} \frac{j!b_{k-1,j}}{(2\pi)^{j+1}(1+i\w)^{j+1}}\sum_{l=0}^{j} \frac{(2\pi\la)^l(1+i\w)^l}{l!}\\
&-e^{-2\pi\la}\sum_{j=0}^{k-1} \frac{j!b_{k-1,j}}{(2\pi)^{j+1}(1+iz)^{j+1}}\sum_{l=0}^{j} \frac{(2\pi\la)^l(1+iz)^l}{l!}\\
&+e^{2\pi i\la z}\cd\frac{1}{\pi^k}\cd\frac{1}{(1+z^2)^k}.
\end{align*}
\end{itemize}

\end{lemma}

\begin{proof}
For item (i), a simple computation yields
\begin{equation}
\ft{G_0}(w,z,t)=\frac{1}{2\pi i(t-z)(t-\w)}.
\end{equation}
Since $\ft{e^{-2\pi|\cd|}}(t)=\frac{1}{\pi(1+t^2)}$, the assertion follows. For item (ii) we have
$$
\sum_{k\geq 1} q^{k-1} \ft{A_k}(\xi) = \sum_{k\geq 1} \frac{q^{k-1}}{\pi^k (1+\xi^2)^k} = \frac{1}{\pi(1+\xi^2)-q} = \frac{1}{\pi(\sqrt{1-q/\pi}^2+\xi^2)}.
$$
By Fourier inversion we obtain
$$
\sum_{k\geq 1} q^{k-1} A_k(|x|) = \frac{e^{-2\pi\sqrt{1-q/\pi}|x| }}{\sqrt{1-q/\pi}}.
$$
Setting $X=2\pi x$ and replacing $q/\pi$ by $q$ we derive the desired assertion.
For item (iii), using that $A_k(x)=e^{-2\pi|x|}p_{k-1}(|x|)$ we have
\begin{align*}
G_k(w,z,\la)&=\left(\int_{-\infty}^{-\la}+\int_{-\la}^{0}+\int_{0}^{\infty}\right)\frac{e^{-2\pi i\w |\la+t|}\1_{(-\infty,0)}(\la+t)+e^{2\pi iz|\la+t|}\1_{[0,\infty)}(\la+t)}{z-\w} A_k(t) dt\\
&=I_1+I_2+I_3.
\end{align*}
A simple computation yields
\begin{align*}
I_1=I_1(w,z)&=\frac{e^{2\pi i\la\w}}{z-\w}\sum_{j=0}^{k-1} (-1)^jb_j^{k-1}\int_{-\infty}^{-\la} t^j e^{2\pi t(1+i\w)}dt\\
&=\frac{e^{-2\pi\la}}{z-\w}\sum_{j=0}^{k-1} \frac{j!b_{k-1,j}}{(2\pi)^{j+1}(1+i\w)^{j+1}}\sum_{l=0}^{j} \frac{(2\pi\la)^l(1+i\w)^l}{l!},
\end{align*}
\begin{align*}
I_2&=I_2(w,z)=\frac{e^{2\pi i\la z}}{z-\w} \sum_{j=0}^{k-1} b_{k-1,j}(-1)^j\int_{-\la}^{0} t^je^{2\pi t(1+iz)}dt\\
&=\frac{1}{z-\w}\left\{-e^{-2\pi\la} \sum_{j=0}^{k-1} \frac{j!b_{k-1,j}}{(2\pi)^{j+1}(1+iz)^{j+1}}\sum_{l=0}^{j} \frac{(2\pi\la)^l(1+iz)^l}{l!}+e^{2\pi i\la z}\sum_{j=0}^{k-1} \frac{j!b_{k-1,j}}{(2\pi)^{j+1}}\cd\frac{1}{(1+iz)^{j+1}}\right\}\\
\end{align*}
and if $z=i$, then 
\begin{equation}
I_2=\frac{e^{-2\pi\la}}{i-\w}\sum_{j=0}^{k-1} \frac{b_{k-1,j}\la^{j+1}}{j+1}.
\end{equation}
Note that $I_2=I_2(w,z)$ is a holomorphic function of $z\in\cp^+$ by dominated convergence and Morera's Theorems. In particular it has no pole at $z=i$. Routine computations also show that
\begin{align*}
I_3=I_3(w,z)&=\frac{e^{2\pi i\la z}}{z-\w} \sum_{j=0}^{k-1} b_{k-1,j}\int_{0}^{\infty} t^je^{2\pi t(iz-1)}dt\\
&=\frac{e^{2\pi i\la z}}{z-\w}\sum_{j=0}^{k-1} \frac{j!b_{k-1,j}}{(2\pi)^{j+1}}\cd\frac{1}{(1-iz)^{j+1}}.
\end{align*}
The result follows. Finally for item (iv), we define
\begin{equation}
H_0(z,x):=e^{-2\pi i\ov{z}|x|}\1_{x<0}+e^{2\pi iz|x|}\1_{x\geq0}, \text{ } (x\in\r)
\end{equation}
and set $H_k(z,x):=H_0(z,\cd)* A_k(x)$ (for $k\geq 1)$. Then, for $\la=0$, $z=s\in\r$, we have
\begin{align*}
H_k(s,0)=\int_{\r} e^{2\pi isx}A_k(x)dx=\ft{A_k}(-s)=\frac{1}{\pi^k}\cd\frac{1}{(1+s^2)^k}.
\end{align*}
On the other hand, using that $A_k(x)=e^{-2\pi|x|}p_{k-1}(|x|)$ and expanding the integral, we obtain
\begin{align*}
H_k(s,0)=\sum_{j=0}^{k-1} \frac{j!b_{k-1,j}}{(2\pi)^{j+1}}\left[\frac{1}{(1+is)^{j+1}}+\frac{1}{(1-is)^{j+1}}\right].
\end{align*}
The claim follows by analytic continuation on $s$.
\end{proof}





We now introduce our main Lemma, which establishes a bridge between the Fourier series and the Nevanlinna factorization of the map $F(z)$ in Theorem \ref{thm:main}. For $k$ an integer, we define the kernel $S_k$:
\begin{equation}
S_k(x):=\frac{1}{v_k}\left(\frac{\sin(\pi x)}{\pi x}\right)^{2(k+1)}
\end{equation}
with $v_k>0$ chosen so that $\ft{S_k}(0)=1$. Note that
\begin{equation}
\ft{S_k}(t) = \frac{1}{v_k} \1_{[-1/2,1/2]}*\cdots*\1_{[-1/2,1/2]}(t);
\end{equation}
the convolution of $2(k+1)$ indicator functions.

\begin{lemma}[The Bridge Lemma]
\label{thm:bridge}
If $(\mu,a)$ is a FS-pair such that $deg(\mu)\leq 2(k+1)$, then
\begin{equation}
\label{eq:bridge}
\lim_{T\ra\infty} \sum_{|\la|\leq T(k+1)} a(\la)G_k(w,z,\la){\ft{S_k}\left(\frac{\la}{T}\right)}=\frac{1}{2\pi^{k+1}i}\int_{\r} \frac{1}{(t-z)(t-\w)}\cd\frac{d\mu(t)}{(1+t^2)^k}.
\end{equation}
uniformly for $z,w \in \cp^+$ in the region
$$R_c:=\{z\in\cp^+;|\Re(z)|\leq1/c \text{ and } \Im(z)\geq c\},$$
for all $c>0$.
\end{lemma}
\begin{proof}
Fix a $c>0$ small and write $\S:=S_k$. Observe that $\ft{\S}$ has compact support with $\supp(\ft{\S})\subset[-k-1,k+1]$ and $\ft{\S}(0)=1$. For $T>0$, let 
$\S_T(x)=T\S(Tx)$
which is an approximation of identity as $T\ra\infty$, and $\ft{\S_T}(t)=\ft{\S}(t/T)$. Take a function $\p\in C^\infty_c(\r)$, with $\p\geq0$, $\supp(\p)\subset[-1,1]$ and $\ft{\p}(0)=1$. For $0<\ep<1$, let $\p_\ep(x)=\p(x/\ep)/\ep$. Then, for fixed $T>4/c$, we define

\begin{equation}
G_{\ep,T}(x)=\left(G_k(w,z,\cd)\ft{\S_T}\right)*\p_\ep(x)
\end{equation}
which belongs to $C^\infty_c(-T(k+1)-1,T(k+1)+1)$. Also,

\begin{equation}
\ft{G_{\ep,T}}(t)=\left(\ft{G_k}(w,z,\cd)* \S_T\right)(t)\ft{\p}(\ep t),
\end{equation}
which converges pointwisely to $\ft{G_k}(w,z,\cd)* \S_T(t)$ when $\ep\downarrow0$.
We also claim that 

\begin{equation}
\label{eq:unif_estimate_1}
\left|\ft{G_k}(w,z,\cd)* \S_T(t)\right|\ll_c\frac{1}{(t^2+c^2)^{k+1}}
\end{equation}
for $z,w\in R_c$, which we prove in the end. Since $(\mu,a)$ is an FS-pair, we have

\begin{equation}
\sum_{|\la|\leq T(k+1)+1} a(\la)G_{\ep,T}(\la)=\int_{\r} \left(\ft{G_k}(w,z,\cd)* \S_T\right)(t)\ft{\p}(\ep t) d\mu(t).
\end{equation}
Because $a(\cd)$ is locally summable, $G_{\ep,T}\ra G_k(w,z,\cd)\ft{\S_T}$ uniformly in $\r$, $\ft{G_{\ep,T}}(t)\ra\ft{G_k}(w,z,\cd)* \S_T(t)$ pointwisely as $\ep\downarrow0$, $||\ft{\p}||_\infty\leq1$ and inequality \eqref{eq:unif_estimate_1}, using the Dominated Convergence Theorem we obtain that
\begin{equation}
\label{eq:intermediate_bridge}
\sum_{|\la|\leq T(k+1)} a(\la)G_k(w,z,\la)\ft{\S}\left(\frac{\la}{T}\right)=\int_{\r} \ft{G_k}(w,z,\cd)*\S_T(t)d\mu(t)
\end{equation}
holds for any $w,z\in R_c$. Note that, for each fixed $c>0$, $T>4/c$ and $w\in\cp^+$, the left-hand side above defines a holomorphic function in $z\in\cp^+$.

We now  show that
\begin{equation}
\label{eq:estimate_ft_G_k}
|\ft{G_k}(w,z,t)|\ll_c\frac{1}{(t^2+c^2)^{k+1}}.
\end{equation}
Indeed, we have
\begin{align*}
|\ft{G_k}(w,z,t)|\ll_c\frac{1}{|t-z|\cd|t-\w|}\cd\frac{1}{(1+t^2)^k}\ll_c\frac{1}{t^2+c^2}\cd\frac{1}{(1+t^2)^k}.
\end{align*}
The family of maps $\ft{G_k}(w,z,\cd)$ for $z,w\in R_c$ is uniformly continuous, uniformly on $z,w$, because each function is Lipschitz with constant independent of $z,w$. Indeed, for $t_1,t_2\in\r$, we have
\begin{align*}
|\ft{G_k}(w,z,t_1)-\ft{G_k}(w,z,t_2)|&\ll\left|\frac{\ft{G_0}(w,z,t_1)}{(1+t_1^2)^k}-\frac{\ft{G_0}(w,z,t_2)}{(1+t_2^2)^k}\right|\\
&\leq\left|\frac{\ft{G_0}(w,z,t_1)}{(1+t_1^2)^k}-\frac{\ft{G_0}(w,z,t_1)}{(1+t_2^2)^k}\right|+\left|\frac{\ft{G_0}(w,z,t_1)}{(1+t_2^2)^k}-\frac{\ft{G_0}(w,z,t_2)}{(1+t_2^2)^k}\right|\\
&\leq\left|\ft{G_0}(w,z,t_1)\right|\left|\frac{1}{(1+t_1^2)^k}-\frac{1}{(1+t_2^2)^k}\right|+\frac{\left|\ft{G_0}(w,z,t_1)-\ft{G_0}(w,z,t_2)\right|}{(1+t_2^2)^k}\\
&\ll \frac{1}{c^2}k|t_1-t_2|+\frac{|t_1-t_2|}{4\pi c^3}\ll_c|t_1-t_2|,
\end{align*}
because
\begin{equation}
|\ft{G_0}(w,z,t_1)-\ft{G_0}(w,z,t_2)|=\frac{1}{2\pi|z-\w|}\left|\frac{1}{t_1-z}-\frac{1}{t_1-\w}-\frac{1}{t_2-z}+\frac{1}{t_2-\w}\right|\leq\frac{|t_1-t_2|}{4\pi c^3}.
\end{equation}
Since $\ft{G_k}(w,z,\cd)$ is uniformly continuous, uniformly for $z,w\in R_c$, we deduce that $\ft{G_k}(w,z,\cd)*\S_T\to \ft{G_k}(w,z,\cd)$ as $T\ra\infty$, uniformly for $z,w\in R_c$. Combining \eqref{eq:intermediate_bridge}, \eqref{eq:unif_estimate_1} and \eqref{eq:estimate_ft_G_k}, and using the Dominated Convergence Theorem, plus the fact that the measure $\mu$ is locally finite, we finally deduce that
\begin{equation}
\lim_{T\ra\infty} \sum_{|\la|\leq T(k+1)} a(\la)G_k(w,z,\la)\ft{\S}\left(\frac{\la}{T}\right)=\frac{1}{2\pi^{k+1}i}\int_{\r} \frac{1}{(t-z)(t-\w)}\cd\frac{d\mu(t)}{(1+t^2)^k},
\end{equation}
uniformly in $z,w\in R_c$, as desired. 

All that remains to be proven is \eqref{eq:unif_estimate_1}. Indeed, for $T>4/c$ and $z,w\in R_c$, we have
\begin{align*}
&\ft{G_k}(w,z,\cd)* \S_T(t)=\frac{1}{2\pi^{k+1}iT^{2k+1}}\int_{\r} \frac{\sin^{2(k+1)}(T\pi s)}{[(t-s)-z][(t-s)-\w][1+(t-s)^2]^k(\pi s)^{2(k+1)}}ds\\
&=\frac{1}{2\pi^{k+1}iT^{2k+1}}\int_{\r} \frac{\sin^k(T\pi(s+i/T))}{[(t-s-i/T)-z][(t-s-i/T)-\w][1+(t-s-i/T)^2]^k[\pi(s+i/T)]^{2(k+1)}}ds
\end{align*}
hence
\begin{align*}
\left|\ft{G_k}(w,z,\cd)* \S_T(t)\right|\ll_c \frac{1}{T^{2k+1}} \int_{\r} \frac{1}{[(t-s)^2+c^2]^{k+1}(s^2+1/T^2)^{k+1}}ds.
\end{align*}
We now use the Residue Theorem to change the contour of integration to $s+ic/2$. Indeed, the map
\begin{align*}
f(s):=\frac{1}{[(t-s)^2+c^2]^{k+1}(s^2+1/T^2)^{k+1}}
\end{align*}
has a pole of order $k+1$ at $s=i/T$, and the residue of $f(s)$ is given by
\begin{align*}
\res{i/T} f(s)=\sum_{\substack{1\leq j,m,n\leq k+1 \\ k+l-j\geq k+1}} \ga_{j,m,n} \frac{\left(t-i/T\right)^j}{T^{2k+1}[(t-i/T)^2+c^2]^{k+l}(2i/T)^{k+m}},
\end{align*}
for some coefficients $\ga_{j,m,n}$, from which we derive
\begin{align*}
\left|\res{i/T} f(s)\right|\ll_c \frac{1}{(t^2+\eta_1 c^2)^{k+1}},
\end{align*}
for some $\eta_1>0$.
On the other hand, the integral over the line $s+ic/2$ gives
\begin{align*}
&\left|\int_{\r} \frac{1}{[(t-s-ic/2)^2+c^2]^{k+1}[(s+ic/2)^2+1/T^2]^{k+1}}ds\right|\\
&\leq\int_{\r} \frac{1}{T^{2k+1}[(t-s)^2+3c^2/4]^{k+1}(s^2+3c^2/16)^{k+1}}\ll_c \frac{1}{(t^2+\eta_2 c^2)^{k+1}},
\end{align*}
where we used the elementary fact about convolutions that if $|f_1(x)|,|f_2(x)|\leq\frac{C}{(1+|x|^2)^k}$ for all $|x|\geq R_1$, then $|f_1*f_2(x)|\leq\frac{C}{(1+|x|^2)^k}$ for $|x|\geq R_2$. The proof is complete.
\end{proof}

\section{Proof of Theorems \ref{thm:main} and \ref{thm:converse_main}}
\begin{proof}[Proof of Theorem \ref{thm:main}: {Necessity}]
Since $a(\cd)$ has exponential growth, there exists $\al>0$ such that $\sum_{\la\in\r} |a(\la)|e^{-2\pi\al|\la|}<\infty$. We claim we can assume that $\al<1$. Indeed, suppose the necessity part of Theorem \ref{thm:main} is proven for $\al<1$. Now, given an arbitrary FS-pair $(\mu,a)$ consider the pair $(y_0^{-1}\mu(y_0\cd),a(\cd/y_0))$ for $y_0>0$ sufficiently large so that $\sum_{\la\in\r} |a(\la/y_0 )|e^{-2\pi\tfrac12|\la|}=\sum_{\la\in\r} |a(\la )|e^{-\pi y_0|\la|}<\infty$. By hypothesis the function
$$
\wt F(z) = \frac12 a(0) + \sum_{\la>0} a(\la/y_0)e^{2\pi i \la z} =  \frac{(z^2+1)^k}{2\pi i}\int_{\r} \frac{tz+1}{t-z}\cd\frac{y_0^{-1}d\mu(y_0t)}{(1+t^2)^{k+1}}+iQ(z),
$$
is a well-defined holomorphic function in $\cp^+$ that belongs to $\NN_{\leq k}-\NN_{\leq k}$, where the second identity above is valid for $\Im z>0$ while the first is valid for  $\Im z>1$. Moreover, $\wt F \in \A\P(\cp^++i)$. By the remarks in the end of Section \ref{sec:neva} (see also \eqref{id:nevfacty0}), it also holds that
$$
\wt F(z) = \frac12 a(0) + \sum_{\la>0} a(\la/y_0)e^{2\pi i \la z} =  \frac{(z^2+y_0^{-2})^k}{2\pi i}\int_{\r} \frac{tz+y_0^{-2}}{t-z}\cd\frac{y_0^{-1}d\mu(y_0t)}{(y_0^{-2}+t^2)^{k+1}}+iQ_1(z),
$$
for some real polynomial $Q_1$ of degree at most $2k+1$. We then let $F(z)=\wt F(z/y_0)$ and observe that $F \in \A\P(\cp^+ +  iy_0)$ and that
$$
F(z) = \frac12 a(0) + \sum_{\la>0} a(\la )e^{2\pi i \la z} = \frac{(z^2+1)^k}{2\pi i}\int_{\r} \frac{tz+1}{t-z}\cd\frac{d\mu(t)}{(1+t^2)^{k+1}}+iQ_1(z/y_0).
$$
The result would then follow.

Therefore, it suffices to consider the case $\sum_{\la\in\r} |a(\la)|e^{-\pi|\la|}<\infty$. We begin by applying the Bridge Lemma \ref{thm:bridge} and multiplying both sides by $(z-\w)$ to get
\begin{equation}
\label{eq:bridge_1}
\lim_{T\ra\infty} \sum_{\la\in\r} a(\la)(z-\w)G_k(w,z,\la)\ft{S_k}\left(\frac{\la}{T}\right)=\frac{1}{2\pi^{k+1}i}\int_{\r} \frac{z-\w}{(t-z)(t-\w)}\cd\frac{d\mu(t)}{(1+t^2)^k},
\end{equation}
which converges uniformly for $z,w$ in compact sets of $\cp^+$. The right-hand side can be written as $H(z)+\ov{H(w)}$, where
\begin{equation}
H(z):=\frac{1}{2\pi^{k+1}i}\int_{\r} \frac{tz+1}{t-z}\cd\frac{d\mu(t)}{(1+t^2)^{k+1}}
\end{equation}
is a holomorphic function in $\cp^+$. Now, we use the explicit form of the map $G_k$, split the limit into the sum and isolate the $z-$terms from the $w-$terms: For $z,w\in\cp^+,z,w\neq i$, by Lemma \eqref{thm:prop_aux_func} (iii) and (iv), the left-hand side of \eqref{eq:bridge_1} can be written as (by simplicity, we write $b_j=b_{k-1,j}$)
\begin{align}
\label{eq:bridge_2}
&F(z)\cd\frac{1}{\pi^k}\cd\frac{1}{(1+z^2)^k}+\ov{F(w)}\cd\ov{\frac{1}{\pi^k}\cd\frac{1}{(1+w^2)^k}}\\
&+\sum_{j=0}^{k-1} \frac{j!b_j}{(2\pi)^{j+1}} \sum_{l=0}^{j} \frac{1}{(1-iz)^{j+1-l}}\ov{\lim_{T\ra\infty} \ga_l(T)}-\sum_{j=0}^{k-1} \frac{j!b_j}{(2\pi)^{j+1}} \sum_{l=0}^{j} \frac{1}{(1+iz)^{j+1-l}}\lim_{T\ra\infty}  \ga_l(T)\\
&+\sum_{j=0}^{k-1} \frac{j!b_j}{(2\pi)^{j+1}} \sum_{l=0}^{j} \frac{1}{(1+i\w)^{j+1-l}}\lim_{T\ra\infty}  \ga_l(T)-\sum_{j=0}^{k-1} \frac{j!b_j}{(2\pi)^{j+1}} \sum_{l=0}^{j} \frac{1}{(1-i\w)^{j+1-l}}\ov{\lim_{T\ra\infty}  \ga_l(T)},
\end{align}
where
\begin{equation}
\label{eq:convergence_series_F(z)}
F(z):=\frac{1}{2}a(0)+\lim_{T\ra\infty} \sum_{\la>0} a(\la)\ft{S_k}\left(\frac{\la}{T}\right)e^{2\pi i\la z},
\end{equation}
which converges uniformly in compact subsets of $\cp^+$. Note that the definition of $F(z)$ does not depend on $y_0$. Moreover,
\begin{align*}
&\ga_0(T):=\frac{1}{2}a(0)+\sum_{\la>0} a(\la)\ft{S_k}\left(\frac{\la}{T}\right)e^{-2\pi \la},\\
&\ga_l(T):=\frac{1}{l!}\sum_{\la>0} a(\la)\ft{S_k}\left(\frac{\la}{T}\right)(2\pi\la)^le^{-2\pi \la}, \text{ for } l\geq1.
\end{align*}
and each of the limits $\lim_{T\ra\infty} \ga_l(T)$ for $l\geq0$, does exist and 
\begin{align}
\label{eq:existence_of_limits}
\lim_{T\ra\infty} \ga_0(T)=\frac{1}{2}a(0)+\sum_{\la>0} a(\la)e^{-2\pi \la}, \text{ and }\\
\lim_{T\ra\infty} \ga_l(T)=\frac{1}{l!}\sum_{\la>0} a(\la)(2\pi\la)^l e^{-2\pi \la}, \text{ for } l\geq1.
\end{align}
Furthermore, since
\begin{align}\label{eq:fourier_series_F}
F(z)=\frac{1}{2}a(0)+\sum_{\la>0} a(\la)e^{2\pi i\la z}
\end{align}
converges uniformly and absolutely for $\Im(z)>1/2$, we can then interchange summation and differentiation:
\begin{equation}
F^{(k)}(z)=\sum_{\la>0} a(\la)(2\pi i\la)^k e^{2\pi i\la z}, \text{ for } k\geq1,
\end{equation}
hence, $\lim_{T\ra\infty} \ga_l(T)=\frac{F^{(l)}(i)}{l!i^l} \text{ for } l\geq0.$ Then we can write \eqref{eq:bridge_2} as $R(z)+\ov{R(w)}$ where
\begin{align*}
&R(z):=\frac{F(z)}{\pi^{k}(1+z^2)^k}+\sum_{j=0}^{k-1} \frac{j!b_j}{(2\pi)^{j+1}} \sum_{l=0}^{j} \frac{1}{(1-iz)^{j+1-l}}\ov{\left(\frac{F^{(l)}(i)}{l!i^l}\right)}\\
&-\sum_{j=0}^{k-1} \frac{j!b_j}{(2\pi)^{j+1}} \sum_{l=0}^{j} \frac{1}{(1+iz)^{j+1-l}}\frac{F^{(l)}(i)}{l!i^l}
\end{align*}
which is holomorphic in $\cp^+\backslash\{i\}$. Therefore, we obtain
\begin{equation}
R(z)+\ov{R(w)}=H(z)+\ov{H(w)}
\end{equation}
for $z,w\in\cp^+\backslash\{i\}$. By analytic continuation, we conclude there exists some $h\in\r$ such that $R(z)=ih+H(z)$. Multiplying both sides by $(z^2+1)^k$ and rearranging terms, we get
\begin{equation}
F(z)\frac{1}{\pi^k}=(z^2+1)^k H(z)+\frac{i}{\pi^k}Q(z)
\end{equation}
where\footnote{We recall the notation $F^*(z)=\ov{F(\ov{z})}$.}
\begin{align*}
&\frac{i}{\pi^k}Q(z):=ih(z^2+1)^k+\sum_{l=0}^{k-1} Q_l(z)-Q_l^*(z)\\
&Q_l(z)=(z^2+1)^k\frac{F^{(l)}(i)}{l!i^l}\sum_{j=l}^{k-1} \frac{j!b_j}{(2\pi)^{j+1}}\cd\frac{1}{(1+iz)^{j+1-l}}.
\end{align*}
Observe that $Q(z)$ is a real polynomial and $Q(z)={h\pi^k} z^{2k}+r(z)$, where $r(z)$ is a real polynomial of degree $\leq 2k-1$. We obtain
\begin{align*}
F(z)=\frac{(z^2+1)^k}{2\pi i}\int_{\r} \frac{tz+1}{t-z}\cd\frac{d\mu(t)}{(1+t^2)^{k+1}}+iQ(z),
\end{align*}
which belongs to the class $\NN_{\leq k}-\NN_{\leq k}$.

Now, from \eqref{eq:fourier_series_F}, we conclude that $F(\cd+i)$ is the uniform limit of trigonometric polynomials. Hence, by \cite[Lemma 8]{G23} $F(\cd+i)\in\A\P(\r)$. Since $F(z)$ is holomorphic and bounded in $\cp^++i$, it follows from \cite[Lemma 11]{G23} that $F(\cd+i)\in\A\P(\cp^+)$, and with Fourier coefficients given by $a(\la)$. Take $c=1$ and (II) is proved. Finally, (III) follows from the finite exponential growth assumption on $a(\cd)$.

\end{proof}

We now prove sufficiency by proving Theorem \ref{thm:converse_main}.

\begin{proof}[Proof of Theorem \ref{thm:converse_main}] Since $F(\cd+ic)\in\A\P(\cp^+)$, it follows from \cite[Lemma 11]{G23} that the limit
\begin{equation}
\e F(\la):=\lim_{T\ra\infty} \frac{1}{2T}\int_{-T+iy}^{T+iy} F(z)e^{-2\pi i\la z}dz
\end{equation}
exist for all $\la\in\r$ and does not depend on $y>c$. Moreover, it holds
\begin{equation}
\sum_{\la\in\r} |\e F(\la)|^2e^{-4\pi\la y}=\e\left[|F(\cd+iy)|^2\right]<\infty
\end{equation}
for any $y>c$.
We now claim that $F(\cd+2ic)$ is bounded. Hence, by \cite[Lemma 12(i)]{G23}, it will follow that $\e F(\la)=0$ for any $\la<0$. In order to prove it is bounded, we will employ the Phragmén-Lindelöf Theorem (as in \cite[Thm 1]{dB}). Let $G(z):=F(z+2ic)$. From Lemma \cite[Lemma 11]{G23}, the map $x\in\r\mapsto G(x)\in\A\P(\r)$, hence it is bounded in $\r$. What remais to prove is that 
\begin{equation}
\label{eq:liminf_Phragmen_Lindelof}
\liminf_{r\ra\infty} \frac{1}{r}\int_{0}^{\pi} \log^+|G(re^{i\theta})|\sin(\theta)d\theta=0,
\end{equation}
where $\log^+(x):=\max\{\log(x),0\}$.
Indeed, since $F(z)$ belongs to the class $\NN_{\leq k}-\NN_{\leq k}$, it admits a Nevanlinna factorization: There exists a unique real signed measure $\mu$ with $\deg(\mu)=2(m+1)$, $m\leq k$, and a real polynomial $Q(z)$ of degree at most $2m+1$, such that
\begin{equation}
F(z)=\frac{(z^2+1)^m}{2\pi i} \int_{\r} \frac{1+tz}{t-z}\cd\frac{d\mu(t)}{(1+t^2)^{m+1}}+iQ(z)
\end{equation}
for $z\in\cp^+$. We can assume, without lost of generality, that $m=k$, otherwise by Proposition \ref{prop:neva} we have $F\in \NN_{\leq m}-\NN_{\leq m}$, and the proof would follow from induction.
Observe that, by splitting in $|t|\leq2|z|$ and $|t|>2|z|$, we obtain $\left|\frac{1+tz}{t-z}\right|\leq\frac{1+2|z|^2}{y}$. Hence $G(re^{i\theta})=O(r^{2k+2})$, and condition \eqref{eq:liminf_Phragmen_Lindelof} follows. 

It remains to prove that $\sum a(\la)\varphi(\la)=\int \ft{\varphi}(t)d\mu(t)$ for any test function $\varphi$. Note that, by linearity, it is enough to show this identity only for antipodal test functions $\p(-x)=\widebar{\p(x)}$, as any test function $\p$ can be written as $\p=u+iv$, where $u$ and $v$ are antipodal. To this end, let
\begin{equation}
H(z):=F(z)-iQ(z)=\frac{1}{2\pi i} \int_{\r} (z^2+1)^k\frac{1+tz}{t-z}\cd\frac{d\mu(t)}{(1+t^2)^{k+1}}.
\end{equation}
On one hand, for $z\in\cp^+$ and $s\in\r$, we have
\begin{align}
\label{eq:identity_H}
H(z+s)+\ov{H(-\ov{z}+s)}&= \int_{\r} P_z(t-s)(1+s^2)^k(1+t^2)\frac{d\mu(t)}{(1+t^2)^{k+1}}\\
&+2k \int_{\r} P_z(t-s)s(1+s^2)^{k-1}(st-s^2+t^2s^2-ts^3)\frac{d\mu(t)}{(1+t^2)^{k+1}}\\
&+\frac{1}{2} \int_{\r} P_z(t-s)h(z,s,t)\frac{d\mu(t)}{(1+t^2)^{k+1}},
\end{align}
where $h(z,s,t)$ is a polynomial in the variables $z,s,t$ with real coefficients such that there is no constant term in $z$, the degree in $t$ is at most 2, and
\begin{equation}
P_z(t):=\frac{z}{\pi i(t^2-z^2)}
\end{equation}
is the Poisson kernel.
Let $\varphi\in\C^\infty_c(-M,M)$ be antipodal, hence $\ft{\varphi}$ is real-valued. From \eqref{eq:identity_H} we obtain
\begin{align}
\label{eq:identity_H_phi}
&\int_{\r} \left[H(z+s)+\ov{H(-\ov{z}+s)}\right]\ft{\varphi}(s)ds=\\
&= \int_{\r} \int_{\r} P_z(t-s)(1+s^2)^k\ft{\varphi}(s)ds(1+t^2)\frac{d\mu(t)}{(1+t^2)^{k+1}}\\
&+2k \int_{\r} \int_{\r} P_z(t-s)s(1+s^2)^{k-1}(st-s^2+t^2s^2-ts^3)\ft{\varphi}(s)ds\frac{d\mu(t)}{(1+t^2)^{k+1}}\\
&+\frac{1}{2} \int_{\r} \int_{\r} P_z(t-s)h(z,s,t)\ft{\varphi}(s)ds\frac{d\mu(t)}{(1+t^2)^{k+1}}
\end{align}
Taking $z=iy$ and using the fact that $P_{iy}$ is an approximation of identity when $y\downarrow0$, from the right-hand side of \eqref{eq:identity_H_phi} and the Dominated Convergence Theorem we obtain
\begin{align}
& \int_{\r} (1+t^2)^{k+1}\ft{\varphi}(t)\frac{d\mu(t)}{(1+t^2)^{k+1}}+2k \int_{\r} t(1+t^2)^{k-1}(t^2-t^2+t^4-t^4)\ft{\varphi}(t)ds\frac{d\mu(t)}{(1+t^2)^{k+1}}\\
&=\int_{\r} \ft{\varphi}(t)d\mu(t).
\end{align}
On the other hand, using that $H(z)=F(z)-iQ(z)$, we can write the left-hand side of \eqref{eq:identity_H_phi} in a different way. Fix an $\Im(z)>c$ and write $z=x+i(c+\eta)$. Since the map $s\mapsto F(s+i(c+\eta))\in\A\P(\r)$, by Bochner's Approximation (see \cite[Proposition 7]{G23}), it can be approximated by a sequence of trigonometric polynomials
\begin{equation}
g_n(s)=\sum_{\la\geq0} \e F(\la)e^{2\pi i\la i(c+\eta)}d_n(\la)e^{2\pi i\la s}
\end{equation}
such that $||g_n-F(\cd+i(c+\eta))||_\infty\ra 0$ as $n\ra\infty$. Here $d_n:\r\ra[0,1]$ is a sequence of functions, each one with finite support and such that
\begin{equation}
\lim_{n\ra\infty} d_n(\la)=\begin{cases}
1 \text{ if } \e F(\la)\neq0\\
0 \text{ if } \e F(\la)=0. \end{cases}
\end{equation}  
Since $Q(z)$ is a real polynomial of degree at most $2k+1$ we can write
\begin{equation}
Q(z+s)=\sum_{l,j=0}^{2k+1} \ga_{l,j} z^l s^j
\end{equation}
where the coefficients $\ga_{l,j}$ are real. Then
\begin{align}
\int_{\r} \left[g_n(x+s)-iQ(z+s)\right]\ft{\varphi}(s)ds&=\sum_{0\leq\la\leq M} \e F(\la)e^{2\pi i\la z}d_n(\la)\varphi(\la)\\
&-i\sum_{l,j=0}^{2k+1} \ga_{l,j} z^l\frac{\varphi^{(j)}(0)}{(2\pi i)^j}.
\end{align}
Since the map $\la\mapsto\e F(\la)$ is locally summable, sending $n\ra\infty$ and using the Dominated Convergence Theorem, it follows that
\begin{align}
\int_{\r} H(z+s)\ft{\varphi}(s)ds=\sum_{0\leq\la\leq M} \e F(\la)e^{2\pi i\la z}\varphi(\la)-i\sum_{l,j=0}^{2k+1} \ga_{l,j} z^l\frac{\varphi^{(j)}(0)}{(2\pi i)^j},
\end{align}
which holds for $\Im(z)>c$.
Now, replacing $z$ by $-\ov{z}$ is the above computation and since $\ft{\varphi}$ is real-valued, we derive
\begin{align}
\int_{\r} \ov{H(-\ov{z}+s)}\ft{\varphi}(s)ds=\sum_{0\leq\la\leq M} \ov{\e F(\la)}e^{2\pi i\la z}\ov{\varphi(\la)}+i\sum_{l,j=0}^{2k+1} \ga_{l,j} (-z)^l\frac{\ov{\varphi^{(j)}(0)}}{(-1)^j(2\pi i)^j}.
\end{align}
Hence, for $\Im(z)>c$, we obtain
\begin{align}
&\int_{\r} \left[H(z+s)+\ov{H(-\ov{z}+s)}\right]\ft{\varphi}(s)ds=\sum_{0\leq\la\leq M} \e F(\la)e^{2\pi i\la z}\varphi(\la)-i\sum_{l,j=0}^{2k+1} \ga_{l,j} z^l\frac{\varphi^{(j)}(0)}{(2\pi i)^j}\\
&+\sum_{0\leq\la\leq M} \ov{\e F(\la)}e^{2\pi i\la z}\ov{\varphi(\la)}+i\sum_{l,j=0}^{2k+1} \ga_{l,j} (-z)^l\frac{\ov{\varphi^{(j)}(0)}}{(-1)^j(2\pi i)^j}.
\end{align}
Since $\varphi$ is antipodal, then $\varphi^{(j)}(-\la)=(-1)^j\ov{\varphi^{(j)}(\la)}$.
From \eqref{eq:identity_H_phi} we obtain
\begin{align}
&\sum_{0\leq\la\leq M} \e F(\la)e^{2\pi i\la z}\varphi(\la)-i\sum_{l,j=0}^{2k+1} \ga_{l,j} z^l\frac{\varphi^{(j)}(0)}{(2\pi i)^j}\\
&+\sum_{-M\leq\la\leq 0} \ov{\e F(-\la)}e^{2\pi i|\la| z}\varphi(\la)+i\sum_{l,j=0}^{2k+1} \ga_{l,j} (-z)^l\frac{\varphi^{(j)}(0)}{(2\pi i)^j}=\\
&= \int_{\r} \int_{\r} P_z(t-s)(1+s^2)^k\ft{\varphi}(s)ds(1+t^2)\frac{d\mu(t)}{(1+t^2)^{k+1}}\\
&+2k \int_{\r} \int_{\r} P_z(t-s)s(1+s^2)^{k-1}(st-s^2+t^2s^2-ts^3)\ft{\varphi}(s)ds\frac{d\mu(t)}{(1+t^2)^{k+1}}\\
&+\frac{1}{2} \int_{\r} \int_{\r} P_z(t-s)h(z,s,t)\ft{\varphi}(s)ds\frac{d\mu(t)}{(1+t^2)^{k+1}}
\end{align}
holds for $\Im(z)>c$. By analytic continuation, equality holds for any $z\in\cp^+$. Moreover, by antipodal splitting, the above equality holds for any $\varphi\in\C^\infty_c(-M,M)$ for arbitrary $M>0$. Taking $z=iy$ and letting $y\downarrow0$, we obtain
\begin{align}
&\sum_{0\leq\la\leq M} \e F(\la)\varphi(\la)-i\sum_{j=0}^{2k+1} \ga_{0,j} \frac{\varphi^{(j)}(0)}{(2\pi i)^j}\\
&+\sum_{-M\leq\la\leq 0} \ov{\e F(-\la)}\varphi(\la)+i\sum_{j=0}^{2k+1} \ga_{0,j} \frac{\varphi^{(j)}(0)}{(2\pi i)^j}=\int_{\r} \ft{\varphi}(t)d\mu(t)
\end{align}
hence
\begin{align}
\e F(0)\varphi(0)+\sum_{0<\la\leq M} a(\la)\varphi(\la)+\ov{\e F(0)}\varphi(0)+ \sum_{-M\leq\la<0} a(\la)\varphi(\la)=\int_{\r} \ft{\varphi}(t)d\mu(t)
\end{align}
or better
\begin{align}
\sum_{\la\in\r} a(\la)\varphi(\la)=\int_{\r} \ft{\varphi}(t)d\mu(t).
\end{align}
The proof is complete.
\end{proof}

\section{Examples of FS-pairs}\label{sec:exe}
We now give some examples of FS-pairs (see also the last section in \cite{G23}). We focus on the measures not contemplated by the result in \cite{G23}, that is, those pairs $(\mu,a)$ such that $\deg(\mu)\geq 3$.

\subsection{The Selberg Trace Formula}

One example of an FS-pair is the {Selberg Trace formula}. If $S$ is a compact hyperbolic surface then, for any even $\p\in\ C^{\infty}_c(\r)$, it holds that
\begin{equation}
\label{eq:Selberg_Trace_f}
\sum_{j\geq0} \ft{\p}\left(\frac{r_j}{2\pi}\right)=\frac{A(S)}{4\pi}\int_{\r} r\ft{\p}(r)\tanh(\pi r)dr+2\pi\sum_{\ga\in G(S)} \frac{\Lambda(\ga)}{N_\ga^{1/2}-N_\ga^{-1/2}}\p(\log(N_\ga)),
\end{equation}
where $A(S)$ is the {hyperbolic surface area} of $S$, $r_j\in\cp$ is a solution for $\la_j=1/4+r_j^2$, where $\{\la_j;j\geq0\}$ is the Laplacian spectrum on $S$, $G(S)$ is the set of closed oriented geodesics on $S$. If $\ga$ is a closed geodesic on $S$, then $N_\ga=e^{l(\ga)}$ is the {norm} of $\ga$ and $\Lambda(\ga)=l(\ga_0)$ the {length} of $\ga$, where $\ga_0$ is the unique oriented prime geodesic satisfying $\ga=\ga_0^m$, for some integer $m\geq1$. Here, $l(\ga)$ is the {hyperbolic length} of the curve $\ga$. For more details, see \cite[Thm. 5.6]{Ber}. Since $\tanh(\cd)$ is odd, by symmetrizing \eqref{eq:Selberg_Trace_f}, we obtain the following identity, which holds for any $\p\in C^{\infty}_c(\r)$:
\begin{align}
\label{eq:Selberg_Trace_symm}
&\frac{2\pi}{A(S)}\sum_{j\geq0} \left[\ft{\p}\left(\frac{r_j}{2\pi}\right)+\ft{\p}\left(-\frac{r_j}{2\pi}\right)\right]-\int_{\r} r\ft{\p}(r)\tanh(\pi r)dr=\\
&\frac{4\pi^2}{A(S)}\sum_{\ga\in G(S)} \frac{\Lambda(\ga)}{N_\ga^{1/2}-N_\ga^{-1/2}}\left[\p(\log(N_\ga))+\p(-\log(N_\ga))\right].
\end{align}
This gives a FS-pair $(\mu,a)$ with
\begin{center}
$\mu=\sum_{j\geq0}( \del_{r_j/2\pi}+\del_{-r_j/2\pi})-r\tanh(\pi r)\d r$\\
$a(\pm\log(N_\ga))=\frac{\Lambda(\ga)}{N_\ga^{1/2}-N_\ga^{-1/2}}$, for $\ga\in G(S)$.
\end{center}
Observe that $\deg(\mu)=3$ because of the contribution of the absolutely continuous part and the fact that $\sum_{j\geq 0} r_j^{-2-\ep}<\infty$ for $\ep>0$, which is a consequence of the Spectral Theorem for compact hyperbolic surfaces (see for instance \cite[Thm. 3.32]{Ber}). More generally, by taking \eqref{eq:Selberg_Trace_symm} for two compact hyperbolic surfaces $S_1$ and $S_2$, and taking the difference, we obtain a crystalline measure:
\begin{align*}
&\frac{1}{A(S_1)}\sum_{j\geq0} \left[\ft{\p}\left(\frac{r_j}{2\pi}\right)+\ft{\p}\left(-\frac{r_j}{2\pi}\right)\right]-\frac{1}{A(S_2)}\sum_{m\geq0} \left[\ft{\p}\left(\frac{s_m}{2\pi}\right)+\ft{\p}\left(-\frac{s_m}{2\pi}\right)\right]=\\
&\frac{2\pi}{A(S_1)}\sum_{\ga\in G(S_1)} \frac{\Lambda_{S_1}(\ga)}{N_{S_1,\ga}^{1/2}-N_{S_1,\ga}^{-1/2}}\left[\p(\log(N_{S_1\ga}))+\p(-\log(N_{S_1\ga}))\right]\\
&-\frac{2\pi}{A(S_2)}\sum_{\ga\in G(S_2)} \frac{\Lambda_{S_2}(\ga)}{N_{S_2,\ga}^{1/2}-N_{S_2,\ga}^{-1/2}}\left[\p(\log(N_{S_2\ga}))+\p(-\log(N_{S_2\ga}))\right],
\end{align*}
where $\{1/4+r_j^2;j\geq0\}$ and $\{1/4+s_m^2;m\geq0\}$ are the spectra of the Laplacian in $S_1$ and $S_2$, respectively. Again the measure $\mu=\sum_{j\leq0}( \del_{r_j/2\pi}+\del_{-r_j/2\pi})-\sum_{m\leq0} (\del_{s_m/2\pi}+\del_{-s_m/2\pi})$ has degree at most $3$. 

\subsection{A crystalline measure involving the sum of three squares function $r_3(n)$} 
Recall that, by Legendre's three square theorem, a non-negative integer $n$ can be written as the sum of three squares of integers if, and only if, $n$ is not of the form $4^a(8b+7)$ for non-negative integers $a,b$. We define the arithmetic function $r_3(n)$ to be the number of ways of representing $n$ as the sum of three squares of integers, this is, $r_3(n):=\#\{m\in\z^3;|m|_2^2=n\}$. We have $r_3(0)=1$ and $r_3(n)=0$ if $n$ is of the form $4^a(8b+7)$. By a result of \cite{CKO}, it holds that, for any positive $\ep$
\begin{equation}
\label{eq:mean_r_3}
\sum_{n\leq x} r_3(n)=\frac{4}{3}\pi x^{3/2}+O(x^{3/4+\ep}).
\end{equation}
Working out on an old example of Guinand, Meyer in \cite[Thm. 4]{Mey} constructed the crystalline measure
\begin{equation}
\mu=\sum_{j\geq1} \frac{\chi(n)r_3(n)}{\sqrt{n}}\left(\del_{\sqrt{n}/2}-\del_{-\sqrt{n}/2}\right),
\end{equation}
and he proved that $\ft{\mu}=-i\mu$. Here, the character $\chi$ is defined by
\begin{equation}
\chi(n)=\begin{cases} -1/2, &\text{ if } n\in\n\backslash4\n\\
4, &\text{ if } n\in4\n\backslash16\n\\
0, &\text{ if } n\in16\n.\end{cases}
\end{equation}
This gives rise to an FS-pair $(\mu,-i\mu)$. Because of \eqref{eq:mean_r_3} and the fact that $r_3(4n)=r_3(n)$, a simple integration-by-parts argument shows that $\deg(\mu)=3$.

\subsection{A generalized family of Guinand's measure} In the last section of \cite{G23}, a construction of Guinand's was generalized.  Let $\eta(z)=q\prod_{n\geq 1} (1-q^n)$ be Dedekind's eta-function, where $q=e^{2\pi i z}$ and $z\in \cp^+$. Consider now the following family
\begin{align*}
\frac{\eta(z)^{24c-2}\eta(4z)^{24c-2}}{\eta(2z)^{48c-5}} & = q^{c}\left(1-(24c - 2)q + (288c^2 - 36c)q^2 + O(q^3) \right) = \sum_{n\geq 0} \al_{n,c} q^{n+c},
\end{align*}
of modular forms  for a real number $c\in[0,1/8]$. Consider then the measure
$$
\mu_c=  \sum_{n\geq 0}\al_{n,c} (\del_{\sqrt{n+c}} + \del_{-\sqrt{n+c}}).
$$
It is shown in \cite{G23} that $\ft \mu=\mu$, that is, $(\mu,\mu)$ is a FS-pair. We  note that $\mu_0=\sum_{\n\in \z} \del_{n}$ produces Poisson's summation and $\mu_{1/9}$ is Guinand's construction in \cite{Gui}, although he did not came up with his construction this way. We also notice that if $c>1/8$, the coefficients  $|\al_{n,c}|$ grow exponentially (but some exceptional values of $c$) and so $(\mu,\mu)$ is not a FS-pair. However, it still possible to generate a summation formula although only for test functions $\p(x)$ which extend analytically and decay sufficiently fast in a strip $|\im z|<b$, for a suitable $b>0$. For $c\in [0,1/8]$, numerical experiments indicate that the coefficients oscillate erratically in the interval $[-1,1]$, which nevertheless would imply that $|\al_{n,c}|\leq 1$. Provably, the Hecke bound shows that  $|\al_{n,c}|\ll_c n^{1/4}$, and so $\deg(\mu_c)\leq 3$ (and conjecturally $\deg(\mu_c)=3$).

\subsection*{Acknowledgements.} The first author acknowledges support from the following funding agencies: The Office of Naval Research GRANT14201749 (award number N629092412126), The Serrapilheira Institute (Serra-2211-41824), FAPERJ (E-26/200.209/2023) and CNPq (309910/2023-4). The second author is supported by CNPq (141446/2023-4).

\subsection*{Competing interest.} The authors have no competing interest to declare.


\end{document}